\documentclass[
  oneside,
  11pt, a4paper,
  footinclude=true,
  headinclude=true,
  cleardoublepage=empty
]{scrbook}

\usepackage{lipsum}
\usepackage[linedheaders,parts,pdfspacing]{classicthesis}
\usepackage{amsmath}
\usepackage{amsthm}
\usepackage{acronym}
\usepackage{tikz}
\usetikzlibrary{arrows,chains,matrix,positioning,scopes}
\usepackage[mathscr]{euscript}
\usepackage{tikz}
\usepackage{amssymb}
\usepackage{mathrsfs}
\usepackage{graphicx}
\usepackage{parskip}

\newcommand{\ad}{\rotatebox[origin=c]{270}{$A$}}

\title{Adjoinable Homology}
\author{Anastasis Kratsios}
\begin{document}
\maketitle

\pdfbookmark[0]{Abstract}{Abstract}
\chapter*{Abstract}
The notion of a duality between two derived functors as well as an extension theorem for derived functors to larger categories in which they need not be defined is introduced.  These ideas are then applied to extend and study the coext functors to an arbitrary coalgebra.  A new homology theory theory is then built therefrom and is shown to exhibit certain duality relations to the Hochschild cohomology of certain coalgebras.  Lastly, a certain exceptional type of coalgebra is introduced and it is used to make explicit connections between this new homology theory and the continuous cohomology of this exceptional algebra's pro-finite dual algebra.

\break

\pdfbookmark[1]{Foreword}{Foreword}
\chapter*{Foreword}

The Hochschild (co)homology of an algebra is a central if not the central tool in non-commutative geometry.  The central objective of this paper is to develop a Homology theory on coalgebras acting dual to the Hochschild cohomology thereon, which as of yet has no homological counterpart.  
The process will be to prove, in chapter 1, three abstract theorems, and introduce the notion of a pseudo-derived functor extending derived functor onto larger categories on which they need not be defined.  \\
Then in chapter 2 to review some of the theory introduced by \textit{"Khaled AL-Takhman"}, in his article \textit{"Equivalences of Comodule Categories for Coalgebras over Rings"} and by \textit{L. Abrams and C. Weibel} in their article \textit{"Cotensor products of modules"}, all unsighted and unproven results derive therefrom or are general knowledge in the field.  \\
Chapter 3 then applies the abstract theory to this setting giving a duality theorem between Hochschild cohomology and adjoinable homology on suitable coalgebra, by means of the pseudo-derived functors of the cohom functor \textit{(which itself is seldom defined)}.  A few properties of this theory are briefly discussed and leading into chapter 4, which is nothing more than a brief review of pro-finite algebras, their relevant relationships to coalgebras and an isomorphism theorem in homology proven by Weibel and Abrams.  \\
Finally The work closes in chapter 5 with the exploration of a certain breed of coalgebras are introduced and lightly explored.  They will provide very concrete computational links between the more classical continuous Hochschild homology theory the coalgebra’s pro-finite dual counterpart.  Finally the paper closes with a special case of the finite theory, and some personal acknowledgements.  
\pdfbookmark[1]{\contentsname}{tableofcontents}

\setcounter{tocdepth}{2} 
\setcounter{secnumdepth}{3} 

\tableofcontents

\newtheorem{pdefn}{Prototypical Definition}
\newtheorem{defn}{Definition}
\newtheorem{lem}{Lemma}
\newtheorem{prop}{Proposition}
\newtheorem{cor}{Corollary}
\newtheorem{thrm}{Theorem}
\newtheorem{dthrm}{Dual Theorem}
\newtheorem{esax}{Eilenberg–Steenrod Axiom}
\newtheorem{dcax}{$\mathfrak{D}$-$\mathfrak{C}$-Axiom}

\newtheorem{ex}{Example}



\chapter{General Abstract Non-sense}
\subsection*{Foreword}
Before moving into specifics three abstracted results will be introduced, a spectral sequence, a related duality theorem, as well the notion of a familly of pseudo-derived functor, extending a family of derived functors to a larger category on which they need not be defined by means of an extension theorem.  

Specifically, a very general spectral sequence is presented and show to converge given a small assumption.  The notion of an F-G-derived duality between two functors is introduced and shown to exist given the assumption made above.  

\section{A Spectral Sequence}
For the duration of this section let $\mathscr{A}$ be an Abelian category with enough projectives and injectives, moreover set $F,$ $G$ to be biendofunctors on $\mathscr{A}$, that is bifunctors such that $F,G :\mathscr{A} \times \mathscr{A} \rightarrow \mathscr{A}$.


\begin{defn}
\textbf{F-G pivots}
Consider two biendofuntors functos F,G on $\mathscr{A}$.  Then a pair of \textit{(possibly the same)} objects $<P,I>$ in $\mathscr{A}$ are called \textbf{F-G pivots} if for every object $A \in \mathscr{A}$ there are isomorphisms $\psi_A$ in $\mathscr{A}$ with: $\psi_A : F(P,A) \cong G(I,A)$.  

\end{defn}
This can be rephrased as follows:  
\begin{prop}
For two bifunctos F-G on $\mathscr{A}$ the pair $<P,I>$ is an $F-G$ pivot if and only if the endofunctors $F(P,-) , G(I,-) : \mathscr{A} \rightarrow \mathscr{A}$ are naturally equivalent.  
\end{prop}

\begin{proof}
By definition.  
\end{proof}

We will use a more refined notion of a pivot from now on.  
\begin{defn}
\textbf{Flipping F-G pivots}

On an abelian category $\mathscr{A}$ both with enough projectives and with enough injectives a pair of objects $<P,I>$ are said to be \textbf{spinning F-G pivots} for two biendofunctors F,G on $\mathscr{A}$ if and only if: 

1) $<P,I>$ are F-G pivots

2) P is projective in $\mathscr{A}$

3) I is injective in $\mathscr{A}$
\end{defn}

\begin{ex}
This notion is not so obscure.  For example consider the category $_RMod_R$ of bimodules on a commutative unital assosiative ring R, then $_RMod_R$ satisfies all the assumptions of being abelian and possessing enough injective and projective objects.  

Now, let P be a finitely generated projective-injective module then $P^{\star } := Hom_R(P,R)$ is again projective.  Moreover there are for ever $R$-bimodule N, there are isomorphisms $P^{\star} \otimes_R N \cong Hom_R(P,N)$.  

If the notion of a projective-injective module is not familiar to the reader, simply consider the case where R is a field.  Then all free, projective modules and injective modules coincide.  
\end{ex}

Hence, pivots between two endofunctors are not imaginary concepts.  

From now on $\mathscr{A}$ is assumed to be abelian both with enough injectives and projectives.  Flipping F-G pivots will be useful in the following way: 

\begin{defn}
\textbf{F-G Flipping Resolution}

An \textbf{F-G flipping resolution} $\mathscr{F}:=<P_{\star},I^{\star}>$ of an object object A in $\mathscr{A}$ is a projective resolution of $... \rightarrow P_n \rightarrow ... \rightarrow P_1 \rightarrow A \rightarrow 0 $ of A, with each $<P_i,I^i>$ being flipping F-G pivots.  
\end{defn}

\begin{lem}
Maintaining the same notation as above, let $\mathscr{F}$ be a flipping resolution of some $A \in \mathscr{A}$.  Then there is a natural isomorphism $\Psi$ between the complexes of endofunctors on $\mathscr{A}$ $\Psi: F(P_i,-) \overset{\cong}{\rightarrow} G(I^i,-)$.  
\end{lem}

\begin{proof}
A flipping F-G pivot $<P_i,I_i>$ is the same as a natural isomorphism $\psi_i: F(P_i,-) \cong G(I_i,-)$.  Now define the isomorphism $\Psi$ in question as the familly of natural isomorphisms $\Psi := \{\psi_i \} : \{F(P_i,-) \rightarrow G(I^i,-)\}_{i}$.  
\end{proof}

Now the abstract form of the first result is introduced, the assumptions made thus far are reiterated in the statement for completeness: 
\begin{prop}
Let $\mathscr{A}$ be an abelian category with enough projectives and injectives.  Moreover, suppose $F,G$ are biendofunctors on $\mathscr{A}$.  Now suppose, A is an object in A which admits a flipping resolution $\mathscr{F}$ of finite length.   Moreover if $G(?,-)$ is right exact in the both inputs, then: 

For all objects $B,C$ in $\mathscr{A}$ there is a convergent spectral sequence: 

$R^pG(L_qF(P_0,B),C) \Rightarrow L_{q-p}F(A,C)$.  

\end{prop}

\begin{proof}
Let $Q_{\star}$ be a projective resolution of C in $\mathscr{A}$.  Now the commutative diagram who's nodes are $C_{\star\star}:=F(P_{\star},Q_{\star})$ and arrows are the differentials corresponding to each resolution $Q_{\star}$ and $P_{\star}$, yields a bicomplex \textit{(again call it)} $C_{\star\star}$ in the usual way.

This is made into a $3^{rd}$ quadrant bicomplex by relabelling the first indices as $p := -p+l$, where $l$ is the length of the \textit{(finite by assumption)} of the flipping resolution $\mathscr{F}$.  

First \textbf{$^{II}E^q_{p,q}$ is calculated}

By the above lemma there is an isomorphisms of complexes of functors $\Psi: F(P_p,-) \rightarrow G(I^p,-)$.  This then gives an isomorphism of complexes $F(P_p, Q_q) \cong G(I^p,Q_p)$.  

Fix p, and compute the homology of $G(I^p,Q_q)$.  Since the $I^q$ are injective, then the covariant right derived functors of each $R^pG(I^{\star},Q_q)$ trivialise for all $p \neq 0$.  Therefore, on the first page: 

$^IE^1_{pq} \cong $
$
\begin{cases} 
G(I^0,Q_q) & \mbox{ if } p=0 \\
0 & \mbox{ if } p \neq 0 
\end{cases}
$.  \\
Which by the natural isomorphism $\psi_0^{-1}$ is isomorphic to: 

$
^IE^1_{pq} \cong
\begin{cases} 
F(P_0,Q_q) & \mbox{ if } p=0 \\
0 & \mbox{ if } p \neq 0 
\end{cases}
$.  

Taking the homology again but with respect to the index q, the spectral sequecne now stabilises on the second page as the $Q_q$ are projective hence acyclic for covariant left derived functors: 

$
^{I}E^2_{pq} \cong
\begin{cases} 
F(P_0,Q_0) & \mbox{ if } p=0 \\
0 & \mbox{ if } p \neq 0 
\end{cases}
$.  

Now \textbf{$^{II}E^q_{p,q}$ is calculated}.  

Calculating first with respect to the index p, fixing q.  Now G is right exact, therefore since the $Q_q$ are projective 
$G(-,Q_q)$ is exact and so it commutes with the homology functors.  

Therefore, $H_{-p}(G(F(P_{\star},B),Q_q))) $

$\cong G(H_{-p}(F(P_{\star},B))),Q_q)$.  Now similarly, the first page of the second filtration of the spectral sequence gives: 

$
^{II}E^1_{pq} \cong
\begin{cases} 
G(H_{-p}(F(P_{\star},B))),Q_q) & \mbox{ if } q=0 \\
0 & \mbox{ if } q \neq 0 
\end{cases}
$.  

Finishing by taking the homology with respect to the indices q, the spectral sequence now stabilises page as:

$
^{II}E^{2}_{pq} \cong
\begin{cases} 
G(H_{-p}(F(P_{\star},B))),Q_0) & \mbox{ if } q=0 \\
0 & \mbox{ if } q \neq 0 
\end{cases}
$, since $G(?,-)$ was right exact in the first input and 
$Q_{\star}$ projectively resolves C then $H_0G(?,Q_{\star}) \cong G(?,C)$.  Hence:  

$
^{II}E^{2}_{pq} \cong
\begin{cases} 
G(H_{-p}(F(P_{\star},B))),C) & \mbox{ if } q=0 \\
0 & \mbox{ if } q \neq 0 
\end{cases}
$.  
\end{proof}

Though it was unnecessary for the proof, for actual applications assuming furthermore, that F is left exact in both imputs yields a more applicable version of the result: 

\begin{cor}
\textbf{The Flipping Spectral Sequence}
Let $\mathscr{A}$ be an abelian category with enough projectives and injectives.  Moreover, suppose $F,G$ are biendofunctors on $\mathscr{A}$.  Now suppose, A is an object in A which admits a flipping resolution $\mathscr{F}$ of finite length.   Moreover if $G(?,-)$ is right exact in the both inputs and $F(-,?)$ is left exact in both inputs then: 

For all objects $B,C$ in $\mathscr{A}$ there is a convergent spectral sequence: 

$R^pG(L_qF(P_0,B),C) \Rightarrow L_{q-p}F(A,C)$.  
\end{cor}

\begin{proof}
As above, moreover since $F(-,B)$ is left exact and $P_{\star}$ projectivly resolves A then $H_0F(P_{\star},B) \cong F(A,B)$ in $\mathscr{A}$.  
\end{proof}

\section{A Duality Theorem}
The goal here is to now in some way relate the derived functors of two endobifunctors G and F on $\mathscr{A}$.  However, a certain weaker version of the notion of acyclicity would be appropriate.  

\begin{defn}
\textbf{F-acyclic of order n}

Let $F: \mathscr{A} \rightarrow \mathscr{B}$ be a contravariant left exact functor between abelian categories, $\mathscr{A}$ having enough injectives.  

An object A in an abelian category $\mathscr{A}$ with enough projectives is said to be \textbf{F-acyclic of order n}, if the left derived functors $L_qF(P_{\star})$ all vanish except \textit{(possibly)} for the case when $q=n$.  
\end{defn}

Heuristically, this is essentially asking that an object behave like a affine n-space does with respect to singular \textit{co}homology.  

Now the second central definition and abstract notion of this text: 

\textit{For the rest of this section assume the setting of the previous corollary}.  

\begin{defn}
\textbf{F-G derived duality of order n}

Let F, G be two biendofunctors on $\mathscr{A}$ \textit{(as specified above)} and A be an object therein.  Then, F and G are said to satisfy a \textbf{derived duality of order n with respect to A} \textit{if and only if} M admits a flipping resolution of finite length n and $F(A,-)$ there some B in $\mathscr{A}$ which is order n $F(A,-)$-acyclic.  
\end{defn}

Now immediately it is concluded from this definition that: 

\begin{cor}
\textbf{The derived duality Theorem}

Let $\mathscr{A}$ be an abelian category with enough projectives and injectives.  Moreover, suppose $F,G$ are biendofunctors on $\mathscr{A}$.  Now suppose, A is an object in A which admits a flipping resolution $\mathscr{F}$ of finite length.   Moreover, if $G(?,-)$ is right exact in the both inputs and $F(-,?)$ is left exact in the both inputs and F and G satisfy a derived duality of order n with respect to A.  Then for every C in $\mathscr{A}$, there is some B in $\mathscr{A}$ and there are natural isomorphisms $\phi_N$ of the derived functors: 

$\phi_{N,\star} : R^{\star}G(L_nF(A,B),C) \cong L_{\star}F(A,C)$.  
\end{cor}

The hypotheses of the theorem do seem plentiful, however a very simple final example will be worked out at the end of this paper, showing the non-vacuousness of the assumptions in a rather familiar setting.

\begin{ex}
\textbf{Derived Duality over fields}

Consider the case where k is a field, C is the k-coalgebra $C:=Hom_k(k,k) \cong k$ \textit{(as k-modules)}.  

Consider a free k-bicomodule $X$ \textit{(which is the same as a free k-bicomodue)}.  Since $k^e \cong k$,  X has free resolution of length 1: $0 \rightarrow k \rightarrow k \rightarrow 0$, hence it has a finite flipping resolution\textit{(since freeness is equivalent to injectivenss which is equivalent to projectiveness over a field k)}. 
\end{ex}

\subsection{Pseudo-derived Functors}
\subsubsection{An small Foreword and some comments}
For the duration of this section assume $\mathscr{A},\mathscr{B},\mathscr{C}$ to be abelian categories, and $\tilde{\mathscr{A}}$ a full subcategory of $\mathscr{A}$ on which there is exists a right exact \textit{(in both inputs)} bifunctor $F:\tilde{\mathscr{A}} \times \mathscr{B} \rightarrow \mathscr{C}$, which is contravariant in the first input and covariant in the second.  

The results and definitions of this section may be applied similar situations where F is left exact or covariant, \textit{mutatis mutandis}.  However, for direct applicability, these assumptions are chosen and made on F.  

\subsubsection{Extending Functors}
The idea of this subsection subsection is a simple but very important one, it is to \textit{"extend"} the aforementioned bifunctor F and more generally its right derived functors $R^{\star}F$, from $\tilde{\mathscr{A}} \times \mathscr{B} \rightarrow \mathscr{C}$ to $\mathscr{A} \times \mathscr{B} \rightarrow \mathscr{C}$.  

This is nothing more than a little trick.  For any object $A \in \tilde{\mathscr{A}}$, notice, that in the definition of a left derived functor, if such a thing were to exist, requires nothing but a resolution of A by injectives $I_A^{\star}$.  Therefore, if there is a resolution $I_A^{\star}$ of $A$, such that the functors $F(I_A^{i},-)$ exist for every $i$ then the right derived functors of the complex of functors $R^{\star}F(I_A^{\star},-)$ exist.  Moreover, since $F(-,-)$ is right exact in the first imput then if $A \in \tilde{\mathscr{A}}$, that is $F(A,-)$ is defined, then for every $B\in \mathscr{B}$, the object $R^{0}F(I_A^{\star},B)$ is isomorphic to the objects $F(A,B)$.  

Moreover, for any $A' \in \mathscr{A}$ admitting an injective resolution of objects entirely in $\tilde{\mathscr{A}}$ and a morphism $\phi: A \rightarrow A'$ in $\mathscr{A}$.  The ususal theory of derived functors implies there exist a unique family of induced morphisms \textit{(of functors)} $\phi^{\star}:R^{\star}F(I_A^{\star},-) \rightarrow R^{\star}F(I_B^{\star},-)$; therefore there is the following definition$/$ proposition:

\begin{prop}
\textbf{Right Pseudo-derived Functors Existence}

Let $\mathscr{A},\mathscr{B},\mathscr{C}$ to be abelian categories, and $\tilde{\mathscr{A}}$ a full subcategory of $\mathscr{A}$ on which there is exists a right exact \textit{(in both inputs)} bifunctor $F:\tilde{\mathscr{A}} \times \mathscr{B} \rightarrow \mathscr{C}$, which is contravariant in the first input and covariant in the second.  

Moreover, assume A is an object in $\mathscr{A}$ which admits an injective resolution $I_A^{\star}$, such that every $I_A^{i}$ is an object in $\tilde{\mathscr{A}}$.  

Then there exist homotopy invariant bifunctors $PR^{\star}F(I_-^{\star},-): \mathscr{A}_{\tilde{\mathscr{A}}} \times\mathscr{B} \rightarrow \mathscr{C}$, where $\mathscr{A}_{\tilde{\mathscr{A}}}$ is the fullsubcategory of $\mathscr{A}$ of objects admitting injective resolutions $I_A^{\star}$, where each $I_A^i$ is an object of $\tilde{\mathscr{A}}$.  

Moreover, when A is an object of the subcategory $\tilde{\mathbb{A}}$ then there are bifunctorial isomorphisms $PR^{\star}F(-,-)\cong R^{\star}F(-,-)$.  
\end{prop}

\begin{proof}
Follows from the discussion preceding the proposition, and the usual properties of right derived functors, for example homotopy invariance.  
\end{proof}

\begin{defn}
\textbf{Right Pseudo-derived Functors}

To each object A and functor F as above, the previously described family of bifunctors $R^{\star}F(I_A^{\star},-)$ evaluated on A in the first imput, are called the \textbf{Right Pseudo-derived Functors} of F \textit{on A}, and are denoted $PR^{\star}F(A,-)$.\\
Moreover, when these exist for all objects A in $\mathscr{A}$ simply call them \textbf{Right Pseudo-derived Functors} of F.  
\end{defn}

Directly by construction Pseudo-derived Functors have the following properties:

\begin{prop}
The construction of a pseudo right-derived functor 

$PR^{\star}F(A,-)$ of an object A is independent of resolution chosen in $\tilde{\mathscr{A}}$.  

\end{prop}
\begin{proof}
Since pseudo-derived functors are defined as derived functors, then they are independed of choice of resolution in $\tilde{\mathscr{A}}$.  
\end{proof}

It should crucially be noted that, an arbitrary resolution of elements in $\mathscr{A}$, does not generally make sence, therefore the above propoisition specifically applies to resolution which give actual, results; and therefore must be in $\tilde{\mathscr{A}}$.  

\begin{prop}
If $0 \rightarrow A \rightarrow B \rightarrow C \rightarrow 0$ is a short exact sequence of objects in $\mathscr{A}$ then there is a long exact sequence of functors: 

$ .. \rightarrow PR^{\star}F(A,-) \rightarrow $ 

$ PR^{\star}F(B,-) \rightarrow PR^{\star}F(C,-) \rightarrow PR^{\star+ + 1}F(A,-) \rightarrow ...$.  
\end{prop}

A particular case of the above discussion, will be the most useful, for the immediate goals of this paper.  Again this follows directly from the above.  

\begin{thrm}
\textbf{The Derived Functor extension theorem}

Let $\mathscr{A},\mathscr{B},\mathscr{C}$ to be abelian categories, and $\tilde{\mathscr{A}}$ a full subcategory of $\mathscr{A}$ on which there is exists a right exact \textit{(in both inputs)} bifunctor $F:\tilde{\mathscr{A}} \times \mathscr{B} \rightarrow \mathscr{C}$, which is contravariant in the first input and covariant in the second.  

Moreover, assume every object A in $\mathscr{A}$ which admits an injective resolution $I_A^{\star}$, such that every $I_A^{i}$ is an object in $\tilde{\mathscr{A}}$.  

Then:

There exist right pseudo-right derived functors $PR^{\star}F(-,-): \mathscr{A}\times \mathscr{B} \rightarrow \mathscr{C}$, coinciding with $R^{\star}F(-,-): \tilde{\mathscr{A}}\times \mathscr{B} \rightarrow \mathscr{C}$ on $\tilde{\mathscr{A}}$.  
\end{thrm}

The assumptions here may seem hefty, however they will all be satisfied in the context to come.  That is, they will provide a way of extending the problematic coext functors from the category of quasi-finite comodules to the entire comodule category over that coalgebra.  Making, the already pertinent theory to be explored, richer and less scarce.

\chapter{Coalgebras and Comodules \textit{(A Review of some literature)}}
For the remainder of the text, unless otherwise specified or until further assumptions are to be made thereon, $R$ is assumed to be an arbitrary unital associative ring.  

\section{Coalgebras and comodules}

\begin{defn}
\textbf{R-coalgebra}

A left \textit{(resp. right)} \textbf{R-coalgebra} is a triple $<C,\epsilon_C , \Delta_C>$ of a left R-comodule together with two $R$-module homomorphism $\epsilon_C: C \rightarrow R$ and $\Delta_C: C \otimes_R C \rightarrow C$, named the counit and comulitplication respectively; satisfying the following identities:

1)  $(1_C \otimes_R \Delta_C) \circ \Delta_C = (\Delta_C \otimes_R 1_C)\circ \Delta_C$ \\
2)  $(1_C \otimes_R \epsilon_C) \circ \Delta_C = 1_C = (\epsilon_C \otimes_R 1_C ) \circ \Delta_C$.  
\end{defn}

Now a morphism $\psi : C \rightarrow D$ of left $R$-coalgebras is simply an $R$-map that respects the identities 1 and 2 above; that is $\psi\circ\Delta_C = \Delta_D$, $\psi \circ \epsilon_C = \epsilon_D$ and $\psi \circ 1_C = 1_D$.  

\begin{defn}
\textbf{R-Comodule}

A left \textit{(resp. right)} \textbf{R-comodule} is a duple $<M,\rho_M >$ of a left $R$-module M and an $R$-linear map $\rho_M : M \rightarrow M \otimes_R C$ satisfying the following two identities: 

1)  $(1_M \otimes_R \Delta_C) \circ \rho_M = (\rho_M \otimes_R 1_M) \circ \rho_M$

2)  $(1_M \otimes_R \epsilon_C) \circ \rho_M = 1_M$

\end{defn}

Similarly, a morphism $\phi: M \rightarrow N$ of left $C$-comodules is one that respects the above identities, that is $\phi\circ\rho_M = \rho_N$ and commutes on the left with $\Delta_C$ and $\epsilon_C$.

As a convention, when it is evident from the context what the tensor product is taken over, the subscript $_R$ will be omitted, likewise with the comultiplication map $\Delta_C$ and its subscript $_C$ and all other specifying subscripts.

Following standard notions, one may verify that these do in fact form categories: 
\begin{defn}
\textbf{The category of R-coalgebra}

\textbf{The category of left} \textit{(resp. right)} \textbf{C-coalgebras} and $C$-algebra morphisms is denoted \textbf{$^CAlg$} \textit{(resp. $Alg^C$)}.  

\end{defn}

and 
\begin{defn}
\textbf{The category of C-comodules}

\textbf{The category of left} \textit{(resp. right)} \textbf{C-comodules} and $C$-module morphisms is denoted \textbf{$^C\mathscr{M}$} \textit{(resp. $\mathscr{M}^C$)}.  

\end{defn}

The goal here will be to study the (co)homology of the latter category as it relates to the former.  Now to do this, certain analogous to familiar functors and their derived counterparts will be briefly reintroduced.  

\subsection{Cotensors, cohomomorphism and their derived functors}
\subsubsection{Cotensor}

As in the case of the Hochshcild (co)homomology theories of 

R-algebras the n$^{th}$ Hocshild homology groups of an algebra A with coefficients in an $A^e$ module M are known to be isomorphic to the n$^{th}$ derived functor $Tor_{A^e}^n(A,M)$.  
Similarly with the n$^{th}$ Hochschild cohomology groups the derived functors $Ext_{A^e}^n(A,M)$.  

The immediate goal here, is to provide a quick dual analogue, dualising in some sence the tensor product $- \otimes_{A^e} A$ , likewise wise with $Hom_{A^e}(A,-)$.  Then study these functor's derived functors and their interplay.  

For an $R$-algebra A, there are many different characterisations of the $A$-tensor product between two \textit{(appropriately sided)} $R$-modules M and N.  One of them is as the cokernel of the $R$-map $f: M \otimes_R A \otimes_R N \rightarrow M \otimes_A N$ mapping $f(m \otimes_R a \otimes_R n) \mapsto ma\otimes_R n - n \otimes_R an$, where $ma$ and $an$ are the \textit{left} and \textit{right} A-actions on M and N respectively.  That is, due to the uniqueness of the cokernel, $M \otimes_A N$ is the unique $A$-module satisfying the exactness of the following sequence:  

$
\begin{tikzpicture}[>=angle 90]
\matrix(a)[matrix of math nodes,
row sep=3em, column sep=1.5em,
text height=1.5ex, text depth=0.25ex]
{ M\otimes_R A \otimes N & M \otimes_R N & M \otimes_A N & 0 \\};
\path[->,font=\scriptsize]
(a-1-1) edge node[above]{$ f $} (a-1-2);
\path[->>,font=\scriptsize]
(a-1-2) edge node[above]{$ coker(f) $}(a-1-3);
\path[->>,font=\scriptsize]
(a-1-3) edge node[above]{$ 0 $}(a-1-4);
\end{tikzpicture}
$\\
Dually building of this concept, the cotensor has been defined as: 

\begin{defn}
\textbf{Cotensor}

The \textbf{cotensor} $M\square_C N$ of a left $C$-comodule M and a right $C$-comodule is defined as the kernel of the map $f: M\otimes_k N \rightarrow M \otimes C \otimes N$, where f is defined as $f:= \rho_M \otimes_R 1_N - 1_M \otimes \rho_N$.  
\end{defn}

The universal property of the kernel then implies that for any right $C$-modules N, L and a morphism between then $\psi : N \rightarrow L$ , there is a unique morphism $\Psi : M \square_C N \rightarrow M\rightarrow_C N$ of $C$-comodules.  Therefore: 

\begin{prop}
For any right $C$-comodule M the cotensor $M \square_C -$ defines a \textit{(covariant)} functor from $\mathscr{M}^C \rightarrow \mathscr{M}^C$.  
\end{prop}

Now the cotensor functor $M\square_C - $: is not always left exact, however in the case where R is a field or more generally when M is flat when considered as an $R$-module then: 

\begin{prop}
If M is flat as an R-module, then $M \square_C -$ is a left exact.  
\end{prop}

\subsubsection{CoHom}
In the case of M,N modules over an R-algebra, there are many characterisations of the functor $Hom_A(M,-)$.  One particular one that its generally overlooked even though it is used ubiquitously is as being the \textit{(unique)} right adjoint to the right exact functor A $- \otimes_R M$.  

Now it is generally known that left adjoints right exact and visa versa.  Therefore it would seem fitting to define a functor, named Cohom as being the left adjoint to the \textit{(usually)} left exact functor $- \square_C M$.   

\begin{defn}
\textbf{Cohom}

For a flat R-module M, the \textbf{cohom} functor $h_D : \mathscr{M}^D \rightarrow \mathscr{M}^C$ is defined to be left adjoint to the left exact functor $- \square_C M : \mathscr{M}^C \rightarrow \mathscr{M}^D$, when it exists.  
\end{defn}

The modules on which this functor exists are evidently of particular interest, therefore we name them and restrict our gaze thereupon: 

\begin{defn}
\textbf{quasi-finite Comdule}

A $C$-comodule $M$ is said to be quasi-finite if $-\square_C M$ admits a left adjoint.  
\end{defn}

Now, the interest in Cohom is immediately evident as a dual analogue to Hom in this context:

\begin{prop}
If M is an R-flat quasi-finite module then: the cohom functor $h_D : \mathscr{M}^D \rightarrow \mathscr{M}^C$ exists and is right exact.  
\end{prop}

\begin{proof}
All left adjoints are right exact.  
\end{proof}

\begin{cor}
If a comodule M, is quasi-finite then it is R-flat.  
\end{cor}
\begin{proof}
If M is quasi-finite then $-\square_C M$ is right adjoint to $h_D(M,-)$, hence it must be left exact.  
\end{proof}

Evidently the cotensor-cohom adjunction may be rephrased as follows: 

\begin{prop}
If X is quasi-finite then for any $C$-comodule N and $D$-comodule M, there are isomorphisms: 

$Hom_{\mathscr{M}^C}(h_D(X,M),N) \cong Hom_{\mathscr{M}^D}(M,X\square_C N)$.  
\end{prop}

\subsubsection{The Derived functors, Cotor and Coext}
Now that the type of exactness has been determined for these functors, the central interest is immediate:  

\begin{defn}
\textbf{Cotor}

If the $C$-comodule M is R-flat then there are left derived functors 

$Cotor_C^n(M,-)$ defined as 

$Cotor_C^n(M,-) := L^n(M \square_C -) : \mathscr{M}^C \rightarrow \mathscr{M}^D$ and $Cotor_C^n(-,M):=L^n(- \square_C M) : \mathscr{M}^C \rightarrow \mathscr{M}^D$.  
\end{defn}

Likewise: 

\begin{defn}
\textbf{Cohom}

If the $D$-comodule M is quasi-finite then there are right derived functors $Cohom_C^n(M,-)$ defined as $Cohom_C^n(M,-) := R^n(h_(M, -)) : \mathscr{M}^D \rightarrow \mathscr{M}^C$.  
\end{defn}

Before concluding this little discussion, the acyclic objects for the former of these two famillies of functors have names: 

\begin{defn}
\textbf{Coflat}

A right \textit{(resp. left)} $C$-comodule F is said to be \textbf{right} \textit{(resp. left)} \textbf{coflat} if for every \textbf{right} \textit{(resp. left)} 

$C$-module it is $Cotor_C^n(M,-)$-acyclic \textit{(resp.$Cotor_C^n(M,-)$-acyclic)}.  
\end{defn}

\subsubsection{Injectors}
Before moving on an extremely important type of $C$-comodule should be defined: 

\begin{defn}
\textbf{Injector}

A quasi-finite right $C$-$D$-bicomodule X, is called an \textbf{injector} \textit{if and only if} the functor $- \square_C X : \mathscr{M}^C \rightarrow \mathscr{M}^D$ for every injective $C$-comodule I, $ I \square_C X$ is an injective $D$-comodule.  
\end{defn}

Following this definition up with an alternative yet equally important characterisation would very shortly be appropriate:  

\begin{prop}
A quasi-finite right $C$-$D$-bicomodule X, is injector \textit{if and only if} the cohom functor $h_D(X,-) :\mathscr{M}^D \rightarrow \mathscr{M}^C$ is exact. 
\end{prop}

For completeness this may be rephrased as follows: 

\begin{cor}
A quasi-finite right $C$-$D$-bicomodule X, is $h_d(-,M)$-acylic for every right $D$-comodule \textit{if and only if} it is an injector.
\end{cor}

\subsection{The Hochschild cohomology theory of coalgebras}
Let C be an R-coalgebra, its enveloping coalgebra $C^e$ is defined as $C^e:=C \otimes_R C^o$, where $C^o$ is its opposite coalgebra.  
\begin{defn}
\textbf{Hochschild Cohomology of a k-coalgebra}

For a right $R$-coalgebra C, its Hochschild cohomology with coefficients in the left $C^e$-bicomodule M $HH^{\star}(M,A)$, is defined to be $HH^{\star}(M,A):= Cotor_{C^e}^{\star}(M,A)$.  
\end{defn}

It would be expected that the Hochschild homology is to also have a definition.  However, this is as yet, not the case as the natural choice of derived functors to work with, the left derived functors of $h_{C^e}(C,-)$ need not be defined and generally are not for an arbitrary $R$-coalgebra $C$, since $C$ need not generally be quasi-finite as an $C^e$-bicomodule.  

\subsection{Notational Convention}
Finally as a general note the Hochschild homology and cohomology of an algebra A will be denoted $Hosh_{\star}(A,-)$ and $Hosh^{\star}(A,-)$, respectively.

\chapter{Adjoinable Homology and Adjoinable-Hochschild Duality}
\section*{Foreword}
The abstract theory has been introduced and the relevant coalgebraic notions have been reviewed.  The goal of this section is to introduce a cohomology theory on coalgebras which, is derived dual to the Hochshild cohomology thereon.  

Since the Hochschild cohomology of a coalgebra C was seen to be identified with the derived functors $cotor_{C^e}(C,-)$.  The central issue here is that the $h_C(C,-)$ functors may not exist.  This concern is completly non-existent when C is quasi-finite as a $C^e$-bicommodule.  

The first objective is to define a homology theory on \textit{all} coalgebras regardless of any quasi-finiteness property.  Moreover, given reasonably manageable coalgebras it is desired that this homology theory behaves dual, in some sense to their Hochschild cohomology.  

This construction will be undertaken in 2 steps, first the homology theory will be introduced for coalgebras which are quasi-finite as bicomodulues over their enveloping coalgebra.  Then, those particularly nice coalgebras will be shown to be abundant enough so that the homology theory in question may be extended to the entire category of coalgebras by the concept of pseudo-derived functors and the derived functor extension theorem.  

This section will close with the consideration of a very manageable class of coalgebras, which are first something like smooth and show that the 

derived-duality theorem hinted towards earlier applies thereon.  

The paper will then close following the next chapter, wherein a particularly convenient type of coalgebra is introduced, the autoenvelopes.  These will greatly ease computations and will provide very concrete links to the continuous Hochschild cohomology of the profintie dual algebra of the coalgebra in question.

\section{Introducing: a new homology theory}
The idea will be glanced at briefly and then properly extended, so that it is usable.  

\begin{defn}
\textbf{Adjoined Homology of a quasi-finite coalgebra}

For any coalgebra C, which is quasi-finite and admits an injective resolution $I_C^{\star}$ of $C^e$-bicomodules each $I_C^i$ being quasi-finite; its \textbf{Adjoined Homology} with coefficients in the left $C^e$-comodule M, denoted $H\ad_{\star}(C,M)$ is defined as:  $H\ad_{\star}(C,M):=coext_{C^e}(C\square_{C^e}C^o,M)$, wher $C^o$ is the opposite $R$-coalgebra of $C$.  
\end{defn}
This may seem a slightly unusual construction and nameing at first and the natural first question that should come to mind is \textit{"do such coalgebras even exist?"}.  Then answer is \textit{"yes, in fact there are enough of them"}.

\section{Pseudo-Derived Functors, Quasi-finite Modules and coext}

This small technical interlude, is fundamental to a more complete theory.  That is, though the adjoined homology theory may and has been described in a straightforward manner on quasi-finite $C$-bicomodules, for a $R$-coalgebra $C$.  It would seem more, useful and interesting if were applicable to any $R$-algebra.  The theory of pseudo-derived functors allows for the theory to transfer over to the general setting in a rather natural way.  

The idea will be as follows, first, to establish that all free comodules are quasi-finite, then to show that there is enough of them.  In turn, the \textit{Derived functor extension theorem} then extends the theory to the entire category.  Following, this extension, a more general definition of adjoinable homology of a coalgebra will then be given, for an arbitrary coalgebra, in a way consistent with and more applicable than the aforementioned presentation.

\subsection{Free bicomodules are quasi-finite}
It is again noted that $R$ is always to be an arbitrary unital associative ring, and $C$ to be an arbitrary $R$-coalgebra, until further mention.  

\begin{lem}
All free $C$-bicomodules are quasi-finite.  
\end{lem}

\begin{proof}
Let I be a set, M and N be $C$-bicomodules and $\mathfrak{U}<I>$ be the free $C$-bicomodule generated on I.  

Then without loss of generality $\mathfrak{U}<I>$ may be identified with the $C$-bicomodule $\underset{i \in I}{\bigoplus} C$.  The left adjoint of $\mathfrak{U}<I> \square_C -$ is constructed as follows: 

$Hom(M,\mathfrak{U}<I> \square_A N) \cong Hom(M,\underset{i \in I}{\bigoplus} C \square_C N)$.  Since $\square_C$ is an additive bifunctor then $Hom(M,(\underset{i \in I}{\bigoplus} C) \square_C N) \cong Hom(M,\underset{i \in I}{\bigoplus} (C \square_C N)) \cong Hom(M,\underset{i \in I}{\bigoplus} N)  \cong \underset{i \in I}{\prod} Hom(M, N) \cong Hom(\underset{i \in I}{\prod} M,N)$.  

Therefore, the left adjoint of $\mathfrak{U}<I> \square_A -$ is identified with $\underset{i \in I}{\prod}$.  

Hence, all free $C$-bicomodules, are quasi-finite.  
\end{proof}

\subsection{The abundance of quasi-finite bicomodules}
It has been establish that free $C$-bicomodules are quasi-finite, now it will be shown that any $C$-bicomodule is in fact \textit{"included"} in a free $C$-bicomodule.  

\begin{lem}
There are enough free $C$-bicomodules, moreover, these are injective $C$-bicomodules.  
\end{lem}

\begin{proof}
For any $R$-coalgebra $C$ the category $^C\mathscr{M}^C$ is dual to the category of rational ${A_C}$-bimodules $_{A_C}Mod_{A_C}^{rat}$ on the profintie dual $A_C$ of $C$ \textit{(This will formally be reviewed in a later chapter, it sufficies to say that $_{A_C}Mod_{A_C}^{rat}$ is a full subcategory of $_{A_C}Mod_{A_C}$ and so an object is projetive in the former, only if it is projective in the latter)}.  

Moreover, since any free ${A_C}$-bimodule is projective then \textit{the duality principle} implies that its dual $C$-bicomodule is injective.  

In a category of modules over a ring, there are always enough free objects, since any free module is rational then inparticular there are enough rational ${A_C}$-bimodules.  Moreover, dualisation preserves freeness, then there are enough free $C$-bicomodules.  
\end{proof}

\begin{cor}
Any $C$-bicomodule admits an injective resolution of quasi-finite $C$-bicomodules.  

Moreover, this resolution may be chosen such that each injective is a free $C$-bicomodule.  
\end{cor}

\begin{proof}
Any $C$-module admits an injective resolution of free 

$C$-bicomodules, moreover any free $C$-bicomodule is quasi-finite.  
\end{proof}

\subsection{Adjoined Homology}
The two final steps are now discussed culminating the discussion to date.  

\begin{prop}
For any $C$-bicomodule M, there the right pseudo-derived functors $PR^{\star}h_C(M,-): ^C\mathscr{M}^C \rightarrow ^C\mathscr{M}^C$ exist, and coincide with the right derived functors $R^{\star}h_C(M,-): ^C\mathscr{M}^C \rightarrow ^C\mathscr{M}^C$ when M is quasi-finite as a $C$-bicomodule.  
\end{prop}

\begin{proof}
Since any module admits an injective resolution by quasi-finite $C$-bicomodules, then the conditions for the \textit{Derived functor extension theorem} are satisfied with $\tilde{\mathscr{A}}$ beng the full subcategory of free $C$-bicomodules.  
\end{proof}

Immediately from this, to \textit{any} $R$-coalgebra a homology theory may be adjoined.  The definition, of the \textit{adjoined homology} of a $R$-coalgebra is now generalised.  

\begin{defn}
\textbf{Adjoined homology of a coalgebra}

Let $R$ be an arbitrary unital associative ring, and $C$ be an $R$-coalgebra and $N$ be a $C^e$-bicomodule, then the \textbf{adjoined homology} of $C$ is defined via the right pseudo-derived functors of $h_{C^e}(-,-)$ as: 

$H\ad_{\star}(C,-):=PR^{\star}h_{C^e}(C\square_{C^e}C^o,-) : ^{C^e}\mathscr{M}^{C^e} \rightarrow ^{C^e}\mathscr{M}^{C^e}$

where $C^o$ is $C$'s opposite $R$-coalgebra. 

Following convention, the $H\ad_{\star}(C,N)$ is called the \textbf{adjoinable homology of C with coefficients in $N$}.  
\end{defn}

For completeness: from now on the pseudo-derived functors of 

$h_C(-,-)$ will be identified with the derived functors $coext_C^{\star}(-,-)$:

\begin{defn}
\textbf{The Pseudo-Coext functors $Pcoext_C^{\star}$}

Let $R$ be an arbitrary unital associative ring, and $C$ be an $R$-coalgebra then the pseudo-derived bifunctors $PR^{\star}(-,-): ^{C}\mathscr{M}^{C} \times ^{C}\mathscr{M}^{C} \rightarrow ^{C}\mathscr{M}^{C}$ are \textit{(not really abusing notation)} named $coext_{C}^{\star}(-,-) :=PR^{\star}(-,-)$.  
\end{defn}
Particularising the above definition to the case where $\star =0$ extends the definition of the cohom bifunctor:

\begin{defn}
\textbf{The Pseudo-Cohom functors $Ph_C^{\star}$}

Let $R$ be an arbitrary unital associative ring, and $C$ be an $R$-coalgebra then the pseudo-derived bifunctor $PR^{0}(-,-): ^{C}\mathscr{M}^{C} \times ^{C}\mathscr{M}^{C} \rightarrow ^{C}\mathscr{M}^{C}$ are \textit{(not really abusing notation)} named $h_C(-,-) :=PR^{0}(-,-)$.  
\end{defn}

\subsection{Three usual results}
Two of the abstract results mentioned in the section on pseudo $-$ derived functors are now rephrased, to this context for clarity and completeness.

\begin{prop}
The constructions of the right pseudo-derived functors $Ph_C^{\star}$, $Pcoext_C^{\star}$ and $H\ad_{\star}(C,-)$ of an $R$-coalgebra C are independent of \textit{quasi-finite} resolution chosen.  
\end{prop}

\begin{proof}
Contextual rephrasing of an above result preceding the 

\textit{"Pseudo-Derived functor extension theorem"}, with $\tilde{\mathscr{A}}$ being the full subcategory $R_Coalg$ consisting of free quasi-finite $R$-coalgebras.  
\end{proof}

\begin{prop}
There is a long exact sequence of functors: 

$.. \rightarrow H\ad_{\star}(C,-) \rightarrow H\ad_{\star+1}(C,-) \rightarrow H\ad_{\star+2}(C,-) \rightarrow ...$.  
\end{prop}
\begin{proof}
Follows by construction of the pseudo-derived functors, as derived functors on special resolution, on which this result holds as in the classical context.  
\end{proof}

\begin{cor}
If $0 \rightarrow M \rightarrow N \rightarrow O \rightarrow 0$ is a short exact sequence of $C^e$-bicomodules then there is a long exact sequence in homology: 

$.. \rightarrow H\ad_{\star}(C,M) \rightarrow H\ad_{\star}(C,N) \rightarrow H\ad_{\star}(C,M) \rightarrow PR^{\star+ + 1}F(C,O) \rightarrow ...$.  
\end{cor}

\begin{proof}
Follows by construction of the pseudoderived functors, as derived functors on special resolution, on which this result holds as in the classical context.  
\end{proof}

This last one is a direct rephrasing of construction:

\begin{prop}
The functors $H\ad_{\star}(-,-)$ and $Pcoext_{C}(-,-)$ are bifunctors.  
\end{prop}

\subsection{Transitional Remarks}
The \textit{adjoinable homology} theory of an $R$-colagebra has been shown to exist regardless of the $R$-coalgebras of choice, unlike what may have been believed at first glance.  

Therefore, a very restricted theory has been extended.  Now therefore, to reap even more results the taken glaze will be again, though only slightly restricted to $R$-coalgebras which are slightly smooth.  

\section{Adjoinable-Hochschild Duality \\
and Dualisable Coalgebras}

\subsection{Dualisable Coalgebras}
The second goal of this paper will now be tackled, that is to show that Adjoinable homology is a natural concept, in many cases \textit{"dual"} to Hochschild cohomology on a fullsubcategory category of $R_Coalg$, of certain $R$-coalgebras called \textit{Dualisable}.    

For any coalgebra C, it was noted that the \textit{pseudo-derived functors} $Pcoext_{C^e}(C,-)$ and $H\ad_{\star}(C,-)$ were expressible via \textit{any} injective resolution, by quasi-finite $C^e$-bicomodules, independently of choice.  Particularly, the existence of a certain specific such resolution admitting special properties will be be key:  

\begin{defn}
\textbf{Dualising Resolution}

A \textbf{Dualising} $A_{\star}: ... \rightarrow A^i \rightarrow ... \rightarrow A^1 \rightarrow M \rightarrow 0$, 

of a $C$-bicomodule M is an $Ph_{C^e}(-,-)$ - $  -\square_{C^e}-$-flipping resolution 

$A_{\star}:= ... \rightarrow A^i \rightarrow ... \rightarrow A^1 \rightarrow M \rightarrow 0$ of M, such that for each $coext_C(-,-)-cotor_C(-,-)$-pivot $<A_i,I^i>$, the $C^e$-bicomodules $A_i$ and $I^i$ are both quasi-finite $C^e$-bicomoduels.  
\end{defn}

The second of the central results of this paper follows has been set up: 

\begin{prop}
If an $R$-coalgebra C, admits a dualising resolution of finite length n and is $Ph_{C^e}(-,-)$-$-\square_{C^e}$ derived dual order n.  

Then for every $C^e$-bicomodule M, there are isomorphisms of right 

$C^e$-bicomodules: 

$HH_{\star}(C,M) \cong Pcoext_{C^e}^{n-\star}(C,M)$.  
\end{prop}

\begin{proof}
All that need be verified is that the functors $cotor_{C^e}(-,-)$ and $Pcoext_{C^e}(-,-)$ exist and do verify the hypothesis of the \textit{derived duality theorem}.  

1)  It was remarked that $-\square_{C^e}-$ is covariant left exact in both inputs.  

2)  Moreover, is $h_{C^e}(-,-)$ is right exact in the both imputs; being covariant in the second  and contravariant in the first input, respectivly.  Therefore, $Ph_{C^e}(-,-)$ is also right exact in the both inputs, with like variances and so must $Ph_{C^e}(-\square_{C^e}-^o,-)$ be in the second input. 

The rest of the assumptions necessary for the use of the \textit{abstract duality theorem} are assumed in the hypothesis.  Therefore, the result now follows.  
\end{proof}

Since, these $R$-coalgebras exibit very particular properties, it seems appropriate to name them:  
\begin{defn}
\textbf{Dualising coalgebra}

A C-coalgebra is said to be \textbf{dualisable} \textit{if and only if} it admits a dualising resolution of finite length $n$ and is $Ph_{C^e}(-,-)$-$-\square_{C^e}-$ derived dual order $n$.  

Moreover if $C$ is dualisable, the integer $n$ above is said the be the \textbf{order} of $C$.  
\end{defn}

Now as a side note by construction: 

\begin{prop}
The category of adjoinable R-coalgebras  is a full subcategory of the category of coalgebras.  
\end{prop}
\begin{proof}
By definition, since all morphisms are admissible.  
\end{proof}

The results and definitions to date are all pronounceable repackaged as:

\begin{thrm}
\textbf{The Adjoined-Hochschild Duality}

For any dualisable $A$-coalgebra C of finite order n.

Then for every right C comodule, there are isomorphisms of $C^e$-bicomodules: 

$HH^{\star}(C,M) \cong H\ad_{n-\star}(C,M)$.  
\end{thrm}

\chapter{Profintie Duals \textit{(A Review of some literature)}}
\subsection{Foreword}
This almost non-existingly-short interlude will provide a lightning fast review of some theory as presented by Weibel and Abrams, in their paper \textit{"Cotensor products of modules"} but with some adjustments to fit this paper's intents.  

The continuous Hochschild cohomology of a profintie algebra, will be rapidly reviewed.  As this is merely a revision, the interested reader is refereed to other articles for proofs or more details.  

The reader already familiar with this material, is encouraged to skip this very brief section.  

\section{Profintie Dual Algebras}
May algebraic objects are in some sense locally finite, particularly coalgebras over a ring $R$.  

\begin{prop}

For any $R$-coalgebra $C$, there exists a collection of subcoalgebras $C^i$, each being finite dimensional, such that the inclusion maps between them and $C$ forms a direct system in $_RCoalg$.
\end{prop}

The construction of a profintie dual of a $R$-coalgebra C is simply the piecing together of its finite $R$-subcoalgebra's dual algebras and their dual structure maps in the category of $R$-algebras.  Formally:

\begin{defn}
\textbf{Profintie dual algebra}

For any $R$-coalgebra $C$, its profintie dual algebra $A_C$ is defined as the inverse system: $<Hom_R(C^i,R),Hom_R(\Delta_i,R,Hom_R(\epsilon_i,R)>$ in the category $_RAlg$, where $<C^i,\Delta_i,\epsilon_i>$ are the finite coalgebras making up $C$.
\end{defn}
These objects form a category as follows:

\begin{prop}
For any ring $R$, the category of profinite $R$-algebras $Pro-R_Alg$, and maps of inverse systems in $R_Alg$ forms a category.  
\end{prop}

This construction implies the following isomorphism of categories: 

\begin{prop}
For any ring $R$, there is an anti-equivalence of the categories $_RCoalg$ and $Pro-R_Alg$.  
\end{prop}

\begin{ex}
For any ring $R$, and any finite dimensional $R$-coalgebra $C$ with finite dimensional dual $R$-algebra $A$, $C$ and $A$ both form trivial direct, and inverse systems in their respective categories, dual to each other via the $Hom_{_RMod}(-,R)$ functor.  This is a special case of the above result.  
\end{ex}

\section{Rational Modules}
In the case where the algebra and its dual coalgebra in question where both finite, there was an equivalence of the categories $_AMod \cong^{-1} \mathscr{M}^C$, where $A$ is the dual algebra of $C$.  This is now generalised to the setting of the pro-objects in question.

\begin{defn}
\textbf{Rational Module}

Let A be a profintie algebra, that is $<A_i,\mu^i, \tilde{\epsilon}^i>$, where $\mu_i$ and $\tilde{\epsilon}^i$ are the structure maps of the finite algebras $A^i$.  

Then a \textbf{rational $A$-module} is an $A$-module, such that for every element $m\ in M$, the module generated by m $Am$ is isomorphic to the quotient module of M by some $A^j$ in the inverse system definite A, that is $Am \cong M/A^j$.  
\end{defn}

\begin{prop}
For any ring $R$ and any profintie $R$-algebra $A$, the collection of all rational $A$-modules and $A$-module homomorphisms forms a full subcategory $_AMod^{rat}$ of $_AMod$ with enough injectives.  
\end{prop}

Mutatis mutandis, one could show:

\begin{cor}
For any commutative ring $R$ and any profintie $R$-algebra $A$, the collection of all rational $A$-bimodules and $A$-module homomorphisms forms a full subcategory $_AMod_A^{rat}$ of $_AMod_A$ with enough injectives.  
\end{cor}

Now the equivalence of categories hinted at above is generalised as follows:  

\begin{thrm}
\textbf{Sweedler’s Theorem}

For any ring $R$ and any $R$-coalgebra $C$ there is an equivalence of categories:  

$_{A_C}Mod^{rat} \cong \mathscr{M}^C$, where $A_C$ is the profintie dual algebra of $C$.  
\end{thrm}

Mutatis mutandis, it then follows \textit{(which will also call by the same name)}: 

\begin{thrm}
\textbf{Sweedler’s Theorem}

For any commutative ring $R$, and any $R$-coalgebra $C$ there is an equivalence of categories:  

$_{A_C}Mod{A_C}^{rat} \cong ^C\mathscr{M}^C$, where $A_C$ is the profintie dual algebra of $C$.  
\end{thrm}

\begin{prop}
For any commutative ring $R$, any profintie $R$-algebra $A$ and any rational $A$-bimoudle $M$, there are $A^i$-bimodules $M^i$, such that $\underset{j}{\varprojlim} M^j \cong M$, when each $M_j$ is considered as an object in the category $_AMod_A^{rat}$.  
\end{prop}

The bimodules denoted $M^j$ above, will be called the \textbf{standard $A^j$-bimodules representing $M$}.  These \textit{"inherit"} an inverse system structure from $A$, when consider in $_AMod_A$.  

\begin{prop}
For any commutative ring $R$, any profintie $R$-algebra $A$ and any rational $A$-bimoudle $M$, the standard $A^j$-bimodules representing $M$ form an inverse system when considered considered as an objects in the category $_AMod_A^{rat}$.  
\end{prop}

\section{Continuous Hochschild Cohomology}

Following the above notion; as one should expect the Hochschild cohomology of a profintie algebra $A$, with values in a rational module $M$, is taken to be the inverse limit of the finite algebras $A^i$ and their corresponding $A^i$-bimoduels $M^i$ making up $A$ and $M$, respectively.  

\begin{defn}
\textbf{Continuous Hochschild Cohomology}

For any commutative ring $R$, any profintie $R$-algebra $A$ and any rational $A$-bimoudle $M$,
the \textbf{continuous Hochschild cohomology} $Hosh_{cnt}^{\star}(A,M)$ of A, with values in M is defined as the inverse limit: 

$Hosh_{cnt}^{\star}(A,M) := \underset{j}{\varinjlim} Hosh^{\star}(A^j,M^j)$, \textit{(where $Hosh^{\star}(A^j,M^j)$ is the usual Hochschild cohomology of the algebra A, with coefficients in the rational $A^i$-bimodule M)}.  
\end{defn}

\begin{ex}
For any commutative ring $R$, any finite dimensional $A$-algebra and any $A$-bimodule $M$, the trivial direct system $<M,1_M>$ implies isomorphisms: 

$Hosh_{cnt}^{\star}(A,M)$ $\cong$ $Hosh_{cnt}^{\star}(A,M)$, and therefore the continuous Hochschild cohomology generalises the usual Hochschild cohomology of an algebra when only considering its coefficients in rational $A$-modules.  
\end{ex}

Finally, the final and possibly result of interest in this review of the literature is the realisation of a result stated earlier:  

\begin{thrm}
\textbf{Abrams and Weibel's Theorem}

For any commutative ring $R$, $R$-coalgebra $C$ and any $C$-bicomodules $M$ and $N$, there are isomorphisms of $R$-bimodules: 

$Hosh_{cnt}^{\star}(A_C,N \otimes_R M) \cong cotor_{C}^{\star}(N,M)$, where $A_C$ is the profinite dual algebra of $C$ and $N$ and $M$ are considered as rational $A_C$-bicomodules in the left hand side of the above isomorphisms.  
\end{thrm}

\chapter{Autoenvelopes and Computations}

Adjoinable-Hochschild duality explicitated a clear relationship between 

the adjoinable homology of a dualisable coalgebra and its Hochschild cohomology. However, both these objects are still obscure, especially the former.  The goal of this section is to provide 

two theorems for computing the adjoinable homology of an adjoinable $R$-bicoalgebra, when $R$ is a commutative unital associative ring by means of the \textit{usual} continuous Hochschild cohomology of its pro-finite dual algebra.

\section{On Autoenvelopes}

Specifically, a certain class of abundant $R$-algebras will prove to be an extremely effective tool for calculating the adjoined homology of an adjoinable $R$-bicoalgebra.  These are the auto-enveloping $R$-bicoalgebras.  

\begin{defn}
\textbf{Autoenvelope}

An $R$-\textbf{autoenvelope} C is a $R$-bicoalgebra, isomorphic to its enveloping $R$-bicoalgebra.  That is: $C^e \cong C$.  
\end{defn}

A small detour is now made, to demystify auto envelopes and their fantastic properties will be shortly studied and used in the remainder of this section. 

Intuitively, autoenvelopes are by nature very large objects, in the following sense.  

\begin{defn}
\textbf{Cocardinality of a $R$-bicoalgebra}

A $R$-bicoalgebra C, is said to be of \textbf{cocardinality $\kappa $,} \textit{if and only if} there exists a unique minimal ordinal $\omega$ and monic $\mu : C \rightarrow \underset{i \in \omega}{\sqcup } R<x>$ and the cardinality $Card(\omega)$ of $\omega$ equals to $\kappa$.  
\end{defn}

\begin{defn}
\textbf{$R$-bicoalgebra of Finite Cocardinality}

A $R$-bicoalgebra C, is said to be of \textbf{finitely cocardinality} \textit{if and only if} C is of cocardinality $n$ for some finite cardinal $n$.   
\end{defn}

\begin{lem}
If R is a commutative ring of characteristic 0 and F is a free, then F is an $R$-autoenvelope \textit{if and only if} F is of cocardinality at least $\aleph_{0}$
\end{lem}

\begin{proof}
Since R is commutative then the coproduct is the tensor product.  Therefore, 
identity F with the coproduct $\underset{i\in I}{\sqcup} R$, where I is some indexing sets.  

The commutativity of R then again implies $\underset{i\in I}{\sqcup} R^o \cong \underset{i\in I}{\sqcup} R$ and so $\underset{i\in I}{\sqcup} R^e = \underset{i\in I}{\sqcup} R \sqcup_R  \underset{i\in I}{\sqcup} R \cong \underset{i\in I}{\sqcup} R \otimes_R R$.  

Since R is of characteristic 0, then $R \otimes_R R \not\cong R$.  Hence, since the forgetful functor is left adjoint to the free functor therefore preserves colimits, and in particular coproduct reduces to: $\underset{i\in I}{\sqcup} R \otimes_R R \cong \underset{i\in I \sqcup I}{\sqcup} R$.  

Now, I is in bijection with $I \sqcup I$ is and only if it is on infinite cardinality.  
\end{proof}

\section{Autoenvelopes and Fields}

Autoenvelopes are related to injectors, coflat comodules and free coalgebras in some of the following ways: 

\begin{prop}
If k is a field then: 

1)  There exist injectors which are autoenvelopes

2)  There exist injectors which are not autoenvelopes

Particularly, free k-coalgebras of infinite cocardinality provide examples of the first case, and ones of finite cocardinaly provide examples of the second.  
\end{prop}

Before embarking on this proof, a little lemma is first presented: 
\begin{lem}
\textbf{A characterisation of injectors in $\mathscr{M}^k$, where k is a field}

If k is a field, A is the k-coalgebra k- $A:=Hom_k(k,k) \cong k$.  Then $X\in \mathscr{M}^k$ is an injector if and only if it has underlying structure of a free k-module.

\end{lem}

\begin{proof}
\textit{(Recall that it was noted above that $\mathscr{M}^k \cong _kMod$).}  

A characterisation of injectors is that for every injective $_A$-module I, X is an injector in $\mathscr{M}^C$ if and only if $I\otimes_k X$ is an injective $_A$-module.  Since A=k, I is injective as a k-module if and only if it is free as a k-module. 

Therefore, if X is an injector then $I\otimes X$ is injective again, since I is free then:

$I \cong \underset{j\in J}{\bigoplus} k$, for some indexing set J.  Since tensoring commutes with direct sums, then $X \otimes_k I \cong X \otimes_k \underset{j\in J}{\bigoplus} k \cong \underset{j\in J}{\bigoplus} X \otimes_k k \cong \underset{j\in J}{\bigoplus} X$.  Therefore $\underset{j\in J}{\bigoplus} X$ must be injective over k, hence free over k, therefore X is the direct summand of a free k-module, hence it must be projective.  However, over a field these two notions of k-freeness and k-projectiveness coincide.  

Conversely, if X is free, then for any injective \textit{(hence free)} module I, $X \otimes_k I$ is again free, therefore injective.  
\end{proof}

\begin{proof}
\textit{( of the proposition)}

\textbf{1)} If since k is a field, then X is an autoenvelope in $_kCoalg$ \textit{if and only if} X has underlying structure of a free $k-bimodule$. Since any free k-coalgebra is a free $k-bimodule$ on an infinite number of generators, and it was noted that a free $k$-coalgebra is an autoenvelope if and only if it is of infinite cocardinality.  Then, $X$ is an autoenvelope \textit{only if} it is of infinite cocardinalty.  Particularly, any free k-coalgebra is an injector, as a k-bicomodule.  

\textbf{2)} Following the above argument, any k-coalgebra of finite coacarinality which is a free $k$-bicomodule, must be an injector what is not an autoenvelope.  Particularly, any free k-coalgebra on finitely many generators, in an injector as a $k$-bicomodule.  
\end{proof}

\begin{cor}
There exist autoenvelopes which are coflat $k$-bicomodules.  Particularly, over a field all free $k$-coalgebras of infinite cocardinality are coflat as $k$-bicomodules.  
\end{cor}

\begin{proof}
Let R be a field, then the autoenvelope $R[x_n]_{n \in \mathbb{Z}}$ is free therefore flat as an $R[x_n]_{n \in \mathbb{Z}}^e$-bimodule.  Identifying 

$Hom_R(R[x_n]_{n \in \mathbb{Z}},R)$ with $Hom_R(R[x_n]_{n \in \mathbb{Z}},R)$ and in turn with its enveloping algebra $Hom_R(R[x_n]_{n \in \mathbb{Z}},R)^e$.  

By duality of the category of algebras and coalgebras corresponding to $Hom_R(R[x_n]_{n \in \mathbb{Z}},R)^e$, $R[x_n]_{n \in \mathbb{Z}}$ is $Hom_R(R[x_n]_{n \in \mathbb{Z}},R)$-coflat.  
\end{proof}

\section{A Partial Characterisation of autoenvelopes \\
on Commutative Rings}

It should be noted that none of the preceding need be true over an arbitrary ring, as free and projective modules differ when R is not a field.  

From now on assume R to be a commutative unital associative ring.  

First, it would be nice to establish, in a few steps, the fact that not all autoenvelopes need be free, injective or even coflat, $R$-bicomodules.  

\begin{lem}
There are rings R, such that coflat comodules are not all injective comodules.  In turn, over these rings injective comodules are not all free.  
\end{lem}

\begin{proof}
For any ring R, $_RMod_R$ is dual to $^R\mathscr{M}^R$.  Since there are rings R, for which flat modules are not all projective; and projective modules are not all free, then by duality the result follows.  
\end{proof}

We now build-up to a general characterisation of autoenvelopes on a ring R, as postulated above.  
So far autoenvelopes may appear to only be huge objects, however this need not be the case, in fact the can be tiny.  

\begin{prop}
There exist rings R, exhibiting autoenvelopes are of finite cocardinalty.  
\end{prop}

\begin{proof}
It is sufficient to exhibit at least one such ring, many more examples can be given by construction.  

Consider the ring $\mathbb{Z}/n\mathbb{Z}$.  Now for any integer n, $\mathbb{Z}/n\mathbb{Z}$ is a $\mathbb{Z}/n\mathbb{Z}$-bimodule.  Since $\mathbb{Z}/n\mathbb{Z}$ is commutative then $\mathbb{Z}/n\mathbb{Z}^o \cong \mathbb{Z}/n\mathbb{Z}$.  Consider the $\mathbb{Z}/n\mathbb{Z}$-linear mapping $\mathbb{Z}/n\mathbb{Z} \otimes \mathbb{Z}/n\mathbb{Z}$, taking the sole generator $1 \otimes 1$ to $\mathbb{Z}/n\mathbb{Z}$'s only generator $1$.  This mapping presents an isomorphism of modules.  Therefore: $\mathbb{Z}/n\mathbb{Z}^e \cong \mathbb{Z}/n\mathbb{Z} \otimes \mathbb{Z}/n\mathbb{Z}^o \cong \mathbb{Z}/n\mathbb{Z} \otimes \mathbb{Z}/n\mathbb{Z} \cong \mathbb{Z}/n\mathbb{Z}$.  

In fact, $\mathbb{Z}/n\mathbb{Z}$ can be given the structure of a $\mathbb{Z}/\mathbb{Z}$-bicomodule, in the usual way since $Hom_{_\mathbb{Z}Mod}(\mathbb{Z}/n\mathbb{Z}, \mathbb{Z}/n\mathbb{Z}) \cong \mathbb{Z}/n\mathbb{Z}$.  

Hence, $\mathbb{Z}/n\mathbb{Z}$ is indeed an autoenvelope.  A straight forward check, shows that $\mathbb{Z}/n\mathbb{Z}$ is of cocardinaliry 1.  
\end{proof}

Autoenvelopes may also enjoy the property of being coflat, over a field this is very common, but in general it need not be the case:

Autoenvelopes actually vary in all sizes, as is the object of the next proposition.  However, a little lemma would first be convenient.  

\begin{lem}
Over a commutative ring.  

The tensor product of two autoenvelopes is again the an autoenvelope, moreover their cocardinlaity is the sum of their cocardinalities.  
\end{lem}

\begin{proof}
Let A, B be autoenvelopes, then $(A \otimes B)^e = (A \otimes B) \otimes (A\otimes B)^o$.  Since the base ring is commutative then any coalgebra is isomorphic to its opposite coalgebra; and so $(A \otimes B)^e \cong (A \otimes B) \otimes (A\otimes B) \cong  (A \otimes A) \otimes (B \otimes B)\cong (A \otimes A^o) \otimes (B \otimes B^o) = A^e \cong B^e$.  Since A and B are autoenvelopes $ A^e \otimes B^e \cong A \otimes B$.  

Therefore in particular, there are short exact sequences, expressing the cocardinality of A, B:

$F(I) \rightarrow A \rightarrow 0$

and

$F(J) \rightarrow B \rightarrow 0$

where $Card(I)$ is the cocardinality of A, likewise with $Card(J)$ and B.  

N is a short exact sequence: 

$F(I)\otimes F(J) \rightarrow A \otimes B \rightarrow 0$ which is minimal in the sense of cocardinality.  

Since the free functor commutes with coproducts then $F(I) \otimes F(J) \cong F(I \sqcup J)$ and so $Card(I \sqcup J) =Card(I) + Card(J)$ is the cocardinality of $A \otimes B$.  
\end{proof}

From this it both of these results at once follow:  

\begin{lem}
Let R is a unital associative ring and I, J are infinite sets.  

Then $\mathfrak{U}(I) \otimes \mathfrak{U}(J) \cong \mathfrak{U}(I \sqcup J)$.  
\end{lem}

\begin{proof}
This was embedded in the above argument, and is generally a property of the free functor.  
\end{proof}

This side note discussion now closes on this note:

\begin{prop}
The full subcategory of $R_Coalg$ consisting of 

$R$-autoenvelopes and $R$-coalgebra homomorphisms is a monoidal subcategory of $R_Coalg$ with product.
\end{prop}

\begin{proof}
The closure follows from the above lemma.  
\end{proof}

\section{The Homology of Autoenvelopes}

The goal of this section is to explicitate a means of calculating the 
adjoined homology of a coalgebra 

by means of the continuous Hochschild cohomology of its pro-finite dual algebra.  This will then be used to explore connections with smooth affine schemes.  

From herein out, assume that $R$ is a commutative unital associative ring.  

\begin{prop}
If $C$ is an autoenvelope, $A_C$ its pro-finite dual algebra and $N$ is a $C$-bicomodule, then there are isomorphisms of $k$-bimodules: 

$ HH^{\star}(C,N) \cong Hosh_{cnt}^{\star}(A_C,C\otimes_k N)$.  

\end{prop}

\begin{proof}
$HH^{\star}(C,N) \cong cotor_{C^e}^{\star}(C,N)$

$\cong cotor_{C^e}^{\star}(C,N)$    \textit{(Since C is an autoenvelope)}  

$\cong Ext_{(A_C)^e}(A_C, C \otimes_k N)$       \textit{(Weibel)}

$\cong Hosh_{cnt}^{\star}(A_C,C\otimes_k N)$.  
\end{proof}

All both all $C$-bicomodules and all rational $A_C$-bimodules have underlying structure of $R$-bimodules.  It should be noted that since their respective structures are forgotten, then there will certainly be a loss in information upon using this approach.  However, it nevertheless produces a very much more explicit understanding of the Homology of larger adjoint algebras.  

Now, via the derived duality theorem, the following and final clarification may be given: 

\begin{thrm}
If $C$ is an autoenvelope which is a dualisable $R$-coalgebra of order $n$ and $N$ is a $C$-bicomodule, then there are isomorphisms of $R$-bimodules:  

$H\ad _{\star}(C,N) \cong Hosh_{cnt}^{n-\star}(A_C, C\otimes_k N)$.

\end{thrm}

\begin{proof}

The Adjoined-Hochschild duality implies 

$H\ad _{\star}(C,N) \cong HH^{n-\star}(C,N)$.  Together with the preceding result the conclusion follows.  

\end{proof}
This immediately demystifies the cohom functor in the aforementioned setting: 

\begin{cor}
If $C$ is an autoenvelope of order $n$ and $N$ is a $C^e$ $bicomodule$, then there is an isomorphism of $R$-bimodules:

$Ph_{C}(C,N) \cong Hosh_{cnt}^{n}(A_C, C\otimes_R N)$.

\end{cor}

\begin{proof}
The assumption that $C$ is an autoenvelope, and the right exactness of $h_{C^e}(C \square_{C^e} C,-)$ together with the case where $\star = 0$ above, imply respectively:
$Ph_{C}(C \square_{C} C,-) \cong Ph_{C^e}(C \square_{C^e} C,-) \cong H\ad _{0}(C,N) \cong Hosh_{cnt}^{n-0}(A_C, C\otimes_k N)$. 

Finally, since $C \square_{C^e} C \cong C \square_{C} C \cong C$, the result follows.  
\end{proof}

If the $R$-coalgebra $C$ in question, is itself quasi-finite as a 

$C^e$-bicomodule, then an extremely explicit calculation of the cohom functor \textit{(which in this case actually exists)} may be given as:  

\begin{cor}
If $C$ is an autoenvelope of order $n$ which is quasi-finite as an $C^e$-bicomodule and $N$ is a $C^e$ $bicomodule$, then there is an isomorphism of $R$-bimodules:

$h_{C}(C,N) \cong Hosh_{cnt}^{n}(A_C, C\otimes_R N)$.
\end{cor}
\begin{proof}
The pseudo-derived functor $Ph_{C}(C,-)$ coincides with the functor $h_{C}(C,-)$ when C is quasi-finite.  
\end{proof}

\section{Finite Coalgebras over fields}
\subsection{Forward}
In the case where C is a autoenvelope, which is a finite dimensional and has finite dimensional dual algebra both over a field, the computations become very explicit, linking together many \textit{(co)}homology theories.  

This final section will bundle it all together in hopes of giving some very computable results.  

\subsection{The Finite theory}

\textit{For the remainder of this section assume $k$ to be a field}.  

In the case where $C$ is a finite dimensional $k$-coalgebra with finite dimensional dual $k$-algebra $A_C$, \textit{Weibel and Abrams} showed, in their work \textit{"Co tensor products of modules"}, that the categories of $C$-bicomodules and rational $A_C$-bimodules equivalent.  Moreover, for any two $C$-bicomodules M and N, the cotor functors $cotor_C^{\star}(M,N)$ are naturally isomorphic to the Hochschild cohomology of the algebra $A_C$ with coefficients in $M\otimes_k N$, where the tensor product is taken over their underlying $k$-vector space and $M\otimes_k N$ is then considered as a rational $A_C$-bimodule via the above equivalence of categories.  In short, in the above setting $cotor_C^{\star}(M,N) \cong Hoch^{\star}(A_C,M\otimes_k N)$, naturally.  

Therefore if $C$ is an autoenvelope then:

\begin{lem}
Let $k$ be a field, $C$ be a finite dimensional $k$-coalgebra which is an autoenvelope of order n, $A_C$ its dual $k$-algebra and $M$ and $N$ are $C$-bicomodules.  Then there are isomorphisms:

$H\ad_{\star}(C,N) \cong Hosh^{n-\star}(A_C,C\otimes_k N)$.  
\end{lem}
\begin{proof}
Since $C$ is an autoenvelope then, $^C\mathscr{M}^C$ is isomorphic to $^{C^e}\mathscr{M}^{C^e}$ and therefore $cotor_{C^e}^{\star}(M,N) \cong cotor_{C}^{\star}(M,N)$, when considering M and M first as $C^e$-bicomoduels then as $C$-bicomodules.  

Setting $M:=C$ above, $HH^{\star}(C,N) \cong cotor_{C}^{\star}(C,N)$.  

From here the proof is a direct consequence the adjoinable-Hochschild duality and the discussion preceding the claim.  
\end{proof}

Next applying the universal coefficients theorem for cohomology implies the following:

\begin{prop}
Let $k$ be a field, $C$ be a finite dimensional $k$-coalgebra which is an autoenvelope of order n, $A_C$ its dual $k$-algebra and $N$ be a $C$-bicomodule. If $A_C$ is hereditary then, there are isomorphisms:

$H\ad_{\star}(C,N) \cong Hom_{_{A_C}Mod}(Hosh_{n-\star}(A_C,A_C),C \otimes_k N) \oplus Ext_{A_C}^{\star}(Hosh_{n-(\star+1)}(A_C),C \otimes_k N)$.  
\end{prop}

\begin{proof}
$Hosh^{n-\star}(A_C,C\otimes_k N)$ can be identified with $Ext_{A_C^e}^{n-\star}(A_C,C\otimes_k N)$.  Furthermore, if $C$ is an autoenvelope then $A_{C^e} = Hom_{_kMod_k}(C^e,k)\cong Hom_{_kMod_k}(C,k) = A_C$.  

Now the rest follows from the universal coefficients theorem and the hereditary assumption made on $A_C$.  
\end{proof}

We close with a very simple example may be computed:
\begin{ex}
Let $k$ be a field, which is an autoenvelope, and $N$ be a $k$-bicomodule. Then, there are isomorphisms:

$H\ad_{\star}(k,N) \cong \begin{cases} N &\mbox{if } \star = 0 \\ 
0 & \mbox{if } else \end{cases} $.  
\end{ex}

\begin{proof}
Considering $k$, both as a $k$-coalgebra and a $k$-bialgebra and noting that the order of any free $k$-autoenvelope on one generator is exactly 1, since it is injective as a $k$-bicomodule, and therefore acyclic and also admits a free resolution of length 1. \\
The little clean up details are just the facts that $Hom_{_kMod_k}(k,-) \cong 1_{_kMod_k} \cong k \otimes_k -$.  

$H\ad_{\star}(k,N) \cong Hom_{_{k}Mod}(Hosh_{1-\star}(k,k),k \otimes_k N) \oplus Ext_{A^e}^{\star}(Hosh_{1-(\star+1)}(k),k \otimes_k N)$.  

Therefore this reduces to:

$H\ad_{\star}(k,N) \cong Hom_{_kMod}(Hosh_{1-\star}(k,k),N) \oplus Ext_{k}^{\star}(Hosh_{-(\star)}(k), N)$.  Ext vanishes for negative indexes, hence: 

$H\ad_{\star}(k,N) \cong Hom_{_kMod}(Hosh_{1-\star}(k,k),N)$.  

Now calculating the Hochschild cohomology in question implies the example.  
\end{proof}

\textbf{Acknowledgements}

I would firstly like to thank my research director and friend Dr. Broer without whom I would most likely not currently be studying within the field of algebraic geometry or have  made any of the progress I have made over the past few years.  Also, I would like to thanks all my other professors who taught me throughout my studies thus far, particularly Dr. Dafni, Dr. Atoyan, Dr. Cornea, Dr. Saliola and Dr. Iovita not only for their excellent classes in which I partook, but also for also their help and advice in life.  

Finally, I would like to particularly thank my parents and familly, to whom I owe everything. 

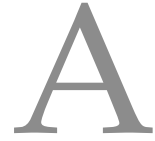
\appendix
\pdfbookmark[1]{References}{References}
\chapter{References}
\section*{Monographs}


-\textbf{Abrams, L, Weibel C.}.  \textit{Cotensor products of modules}. Arxiv. (1999), Web.  

-\textbf{Arapura, D.}.  \textit{Algebraic Geometry over the complex numbers}.  Springer, 2012. Print.  

-\textbf{Borel, Armand}.  \textit{Linear Alebraic Groups}.  New Jersey: Princeton University Press  (1991). Web.  

-\textbf{Broer, Abraham}.  \textit{Introduction to Commutative Algebra}.  Montreal: Université de Montréal (2013). Web.  

-\textbf{Cameron, Peter J.}.  \textit{Notes on Classical Groups}.  London: School of Mathematical Sciences (2000). Web.  

-\textbf{Cartan, H. and Eilenberg, S.  }.  \textit{Homological algebra}.  New Jersey.  Princeton University Press, 1956. Print.    

-\textbf{Cibotaru, D.}.  \textit{Sheaf Cohomology}.  Indiana.  Notre Dame University, 2005. Web.    

-\textbf{J. Cuntz, D. Quillen}.  \textit{Algebra extensions and nonsingularity}, AMS (1995), Web. 

-\textbf{Cuntz, J, Skandalis, G, Tsygan, B.}.  \textit{Cyclic Homology in \\
Non-Commutative Geometry}, Springer (2004), Print. 

-\textbf{Demazure, M., Gabriel P}.  \textit{Introduction to Algebraic Geometryand Algebraic Groups}. Ecole Polytechnique, France (1980), Print.  

-\textbf{Eilenberg, Samuel; Mac Lane, Saunders}.  \textit{On the groups of H($Pi$,n). I}.  Annals of Mathematics. (1953).  Print.  

-\textbf{Gelfand, Manin}.  \textit{Methods of Homological Algebra, 2nd ed.}.  Springer, 2000.  Print.  

-\textbf{Godement, R.} \textit{Topologie algébrique et théorie des faisceaux}.  Paris, Hermann, 1973.   Print.  

-\textbf{Grothendieck, A.}.  \textit{Éléments de géométrie algébrique (I-III)}. Springer, 1971.  Print. \\ -\textbf{Grothendieck, A.}.  \textit{Sur quelques points d'algèbre homologique}.  The Tohoku Mathematical Journal, 1957.  Print.  

-\textbf{Guccione, J. A \textit{and} Cuccione J. J.}.  \textit{The theorem of excision for Hoschild and cyclic homology}.  Journal of Pure and Applied Algebra 106.  Buenos Ares. Argentina, 1996.  Web.  

-\textbf{Harari, David}.  \textit{Schemes}.  Beijing: Tsinghua University (2005). Print.  

-\textbf{Hartshorne, Robin}.  \textit{Algebraic Geometry}.  New-York: Springer (1977). Print.  

-\textbf{G. Hochschild}.  \textit{On the cohomology groups of an associative algebra}, J. Amer. Math. Soc. (1945), Web. 

-\textbf{Humphreys, James E.}.  \textit{Introduction to Lie Algebras and Representation Theory}. New York: Springer-Verlag, (1972). Print.  

-\textbf{F. Ischebeck}.  \textit{Eine Dualitat zwischen den Funktoren Ext und Tor}, J. Algebra 11 (1969), Web. 

-\textbf{Jantzen, Jens Carsten}.  \textit{Representations of Algebraic Groups}.  London: Academic Press Inc. (London) LTD. (1987). Print.  

-\textbf{Johnson S.}.  \textit{Summary: Analytic Nullstellensatz Pt 1.}.  Vancoover. 2012. Web.

-\textbf{Kleshchev, Alexander}.  \textit{Lectures on Algebraic Groups}.  Oregon: University of Oregon (2008). Web.  

-\textbf{Knapp, Anthony W.}.  \textit{Lie Groups Beyond an Introduction - 2nd Edition}.  Boston: Birkhäuser (2004). Print.  

-\textbf{J.C. McConnell, J.C. Robson}.  \textit{Noncommutative Noetherian rings}, Ann. of Math. (1945), Web.

-\textbf{Milne, James}.  \textit{Basic Theory of Affine Group Schemes (AGS)}.  (2012) Web.  

-\textbf{AL-Takhman, Khaled}.  \textit{Equivalences of Comodule Categories for
Coalgebras over Rings}, (2001), Web.  

-\textbf{T.Y. Lam}.  \textit{Lectures on modules and rings}, Springer-Verlag (1999), Web. 

-\textbf{Loday, J.L.}.  \textit{Cyclic Homology}.  Berlin Heidenberg. Springer-Verlang.  (1998).  Print.  

-\textbf{D. Quillen}.  \textit{Algebra cochains and \\
cyclic cohomology}, Inst. Hautes Etudes Sci. Publ. Math. 68 (1988), Web. 

-\textbf{Royden H., Fitzpatrick P.}.  \textit{Real Analysis - 4th Edition}.  Boston: Pearson Hall (2010). Print.  

-\textbf{W. F. Shelter}.  \textit{Smooth algebras}, J. Algebra 103, (1986), Web. 

-\textbf{Springer, Tony A.}.  \textit{Linear Algebraic Groups}.  Budapestlaan: Modern Birkhäuser (2008). Print.  

-\textbf{M. van den Bergh}.  \textit{A relation between Hochschild homology and cohomology for Gorenstein rings}. Proc. Amer. Math. Soc. 126 (5) (1998), 1345-1348. Erratum: Proc. Amer. Math. Soc. 130 (9) (2002), 2809-2810 (electronic).  Web. 

-\textbf{C. Weibel}.  \textit{An introduction to homological algebra}, Cambridge University Press (1995), Web. 

-\textbf{Wodzicki, M. }.  \textit{Excision in Cyclic Homology and in Rational Algebraic K-theory}. Annals of Mathematics 129 (1989), 591-639.  Web.  

-\textbf{C. Weibel}.  \textit{An introduction to homological algebra}, Cambridge University Press (1995), Web. 

\tiny{\tiny{ICXRNIKA}}

\end{document}